\documentclass[reqno]{amsart}

\usepackage[T1]{fontenc}
\usepackage[utf8]{inputenc}
\usepackage{lmodern}
\usepackage{microtype}

\usepackage{comment}
\usepackage{enumitem}
\setlist[itemize]{leftmargin=*, itemsep=.25em}
\setlist[enumerate]{leftmargin=*, itemsep=.25em}

\usepackage{amsmath, amsthm, amssymb, amsfonts, mathtools, bm, mathrsfs}
\numberwithin{equation}{section}

\newtheorem{theorem}{Theorem}[section]

\newtheorem{proposition}[theorem]{Proposition}
\newtheorem{corollary}[theorem]{Corollary}
\theoremstyle{definition}
\newtheorem{definition}[theorem]{Definition}
\newtheorem{example}[theorem]{Example}
\theoremstyle{remark}

\usepackage{graphicx}
\usepackage{subcaption}
\usepackage{booktabs}
\usepackage{tikz}

\newcommand{\Iverson}[1]{\left[\,#1\,\right]}

\usepackage{hyperref}

\hypersetup{
    pdfauthor={Wanderson Matos (E-mail: wanderson.matos.e.silva@gmail.com; ORCID: 0009-0009-7834-611X)},
    pdftitle={Triangular Base Inversion for Lambda-Recursive Polynomial Families},
    pdfsubject={Algebraic Combinatorics},
    pdfkeywords={lambda-recursive polynomials, Clenshaw algorithm, triangular inversion},
    colorlinks=true,
    linkcolor=black,
    urlcolor=black,
    citecolor=black
}

\usepackage[nameinlink,capitalise,noabbrev]{cleveref}

\crefname{theorem}{Teorema}{Teoremas}
\crefname{lemma}{Lema}{Lemas}
\crefname{proposition}{Proposição}{Proposições}
\crefname{corollary}{Corolário}{Corolários}
\crefname{definition}{Definição}{Definições}
\crefname{remark}{Observação}{Observações}
\crefname{example}{Exemplo}{Exemplos}
\crefname{equation}{Equação}{Equações}
\crefname{figure}{Figura}{Figuras}
\crefname{table}{Tabela}{Tabelas}

\begin{document}
\title[Basis inversion in $\lambda$-recursive families]%
{Basis inversion in $\lambda$-recursive families: triangular kernels and polynomial basis changes}
\author{Wanderson Matos}
\address{Independent Author, Cáceres--MT, Brazil}
\email{wanderson.matos.e.silva@gmail.com}
%\subjclass[2020]{Primary 05A, 15A; Secondary 13F}
%\keywords{Sequências recorrentes bidimensionais, Famílias polinomiais, Inversão algébrica de base, Determinantes simbólicos}

\begin{abstract}
Polynomial families $\{f_n(x)\}_{n\ge 0}$ over a commutative ring $R$ are encoded by a triangular array of order $m$, supported on indices $0\le b\le \lfloor n/m\rfloor$, via
\[
f_{n}(x)=\sum_{b=0}^{\lfloor n/m\rfloor}\lambda_1(n,b)\,x^{\,n-mb},
\]
where $\lambda_1$ is the \emph{direct kernel} and satisfies $\lambda_1(n,b)=0$ for $b<0$ or $b>\lfloor n/m\rfloor$. On the combinatorial side, reindexing yields a central change-of-basis result: under a simple discrete orthogonality condition, there exists a unique \emph{inverse kernel} $\lambda_3$ (a triangular array of order $m$) such that
\[
x^n=\sum_{b=0}^{\lfloor n/m\rfloor}\lambda_3(n,b)\,f_{n-mb}(x).
\]
This reindexing mechanism produces explicit change-of-basis formulas between two families, even with distinct step sizes $m_1$ and $m_2$; once the inverse kernel is known, the connection coefficients are obtained from a single universal sum indexed by the triangular support.

On the algebraic side, $\lambda_1$ determines a lower Hessenberg matrix $\mathcal M_{(n,k)}$, called the \emph{algebraic expansion matrix}, whose determinant governs inversion: for each fixed $n$ and $1\le k\le\lfloor n/m\rfloor$, the coefficient $\lambda_3(n,k)$ admits a closed determinantal formula in terms of the direct kernel $\lambda_1$.

From direct kernels, the notion of a lambda-recursive sequence of order $m$ is introduced, specified by a principal factor $(p_n)$ and auxiliary factors $(h_{(n,k)})$. In this framework, $\det(\mathcal M_{(n,k)})$ satisfies a recurrence whose consequences enable direct computation of coefficients in polynomial basis changes. In particular, for lambda-recursive coefficients of order $m$, recurrences are derived for the inverse kernel $\lambda_3(n,k)$,
\[
\lambda_{3}(n,k)=
\begin{cases}
   \dfrac{1}{\lambda_{1}(n,0)}, & k=0,\\[8pt]
    A(n,k)\,\lambda_{3}(n-1,k)
    +\dfrac{h_{s}}{p_{s}}\;\lambda_{3}(n-1,k-1), & 1\le k\le w_m(n),
\end{cases}
\]
where $A(n,k)$ is defined in terms of the boundary factor; analogously, the change-of-basis coefficients satisfy
\[
z(n,k)=
p^{(1)}_{n}A_\mu(n,k)\,z(n-1,k)
+p^{(1)}_{n}\frac{h^{(2)}_{s}}{p^{(2)}_{s}}\,z(n-1,k-1)
-h^{(1)}_{n}\,z(n-m,k-1),
\]
exhibiting a bidirectional structure.

Several classical families (Chebyshev, Legendre, Hermite, Laguerre, Fibonacci, Lucas, among others) fit naturally into this lambda-recursive framework, so that their connection coefficients and change-of-basis coefficients are unified by the same computation on triangular arrays. The same structure supports structured versions of Clenshaw-type schemes, interpolation in adapted bases, and further algebraic reinterpretations, to be developed in future work.
\end{abstract}

\maketitle

\section{Introduction}

Triangular arrays of coefficients between polynomial bases arise in numerous contexts within algebraic combinatorics and the theory of special polynomials, including connection coefficients between orthogonal families, basis changes between $\{x^n\}$ and classical families, and identities encoded in triangular arrays such as Riordan matrices. In these settings, the essential information is carried by triangular matrices of coefficients that perform inversions or basis changes and satisfy highly structured discrete recurrences; see, for instance, \cite{Riordan1968,Shapiro1991,HeSprugnoli2009} for the Riordan-array viewpoint, and \cite{Ronveaux1995,Godoy1997,Area1998,Lewanowicz2000,Maroni2008} for recurrences of connection coefficients between orthogonal polynomials.

In the classical case of orthogonal polynomials in one variable, the standard starting point is a three-term recurrence in $n$ for a sequence $\{p_n\}$, from which connection formulas with the monomial basis and between different special families are deduced; cf.\ \cite{Szego1975,Chihara1978,Ismail2005,KoekoekLeskySwarttouw2010}. In this work, a complementary viewpoint is adopted: the starting point consists directly of polynomial families
\[
f_{n}(x)=\sum_{b=0}^{w_m(n)}\lambda_1(n,b)\,x^{\,n-mb},\qquad w_m(n)=\big\lfloor n/m\big\rfloor,
\]
whose exponents follow a triangular pattern of order $m\ge1$. For each $n$, the terms of $f_{n}(x)$ are powers $x^{\,n-mb}$ for $0\le b\le w_m(n)$, which organizes the coefficients $\lambda_1(n,b)$ into a triangular array. This triangular support $n-mb$ is treated systematically (where $m=1$ represents the classical subcase), and the theory is organized around two doubly indexed kernels: the direct kernel $\lambda_1$ and the inverse kernel $\lambda_3$.

It is shown that the inversion
\[
x^n=\sum_{b=0}^{w_m(n)}\lambda_3(n,b)\,f_{n-mb}(x)
\]
is equivalent to discrete orthogonality relations between $\lambda_1$ and $\lambda_3$. Furthermore, the inverse kernel is unique whenever it exists, and its coefficients admit a determinantal description in terms of a lower Hessenberg matrix associated with the family. In particular, a first set of \emph{direct} formulas is obtained for the inversion coefficients $x^n\leftrightarrow f_n$ in terms of triangular determinants depending on $n$.

The second axis of this article concerns the class of \emph{lambda-recursive families of order $m$}, in which the direct kernel satisfies a two-dimensional recurrence of the form
\[
\lambda_1(n,k)=p_n\,\lambda_1(n-1,k)\;-\;h_{(n,k)}\,\lambda_1(n-m,k-1),
\]
with initial data on the boundary $k=0$. Under natural assumptions, it is shown that the boundary sequence $\lambda_1(n,0)$ is determined by the initial values $\lambda_1(0,0),\dots,\lambda_1(m-1,0)$ and by the principal factor $(p_n)$, and closed recurrences, compatible with the triangular support, are derived for the coefficients of the inverse kernel $\lambda_3(n,k)$. Moreover, when two lambda-recursive families $f$ and $g$ have compatible parameters, it is shown that the change-of-basis coefficients $f\leftrightarrow g$ also satisfy explicit triangular recurrences, obtained mechanically from $\lambda_1$, $\mu_1$, and the corresponding inverse kernels. These local recursive formulas provide a systematic mechanism for computing inversions and basis changes for arbitrary order $m$.

Several classical systems (including Chebyshev, Legendre, Hermite, Laguerre, Fibonacci, and Lucas) naturally fit into this lambda-recursive framework with triangular support $n-mb$, allowing different types of connection and inversion coefficients to be unified and organized into triangular arrays of order $m$. The kernel-based approach highlights which data (boundary, principal factor, and auxiliary factors $h_{(n,k)}$) control the inversion and how compatible deformations preserve the change-of-basis pattern. The same structure suggests structured versions of Clenshaw-type schemes, as well as applications to interpolation in adapted bases and algebraic reinterpretations of the algebraic expansion matrix, topics that will be explored in future work.

The organization of the article is as follows. In Section~\ref{sec2}, notation is fixed, the triangular support of order $m$ is introduced, and staircase reindexation identities are established, culminating in a general inversion theorem between $\{x^n\}$ and a triangular family. In Section~\ref{sec3}, the algebraic expansion matrix is defined, lambda-recursive sequences of order $m$ are introduced, and the determinantal recurrence for the lower Hessenberg matrix is obtained, together with the boundary theorem and closed recurrences for $\lambda_3(n,k)$. Next, in Section~\ref{sec:inv}, these results are specialized to the basis inversion problem for lambda-recursive families and to recursive formulas for the change-of-basis coefficients between two families. Finally, Section~\ref{sec:conc} gathers conclusions and indicates algebraic and algorithmic developments for future work.

\section{Preliminaries}\label{sec2}

\subsection{Notation}

\subsection*{General conventions}

Throughout this work, the following are fixed:
\begin{itemize}
    \item $\mathbb{N}=\{0,1,2,\dots\}$ and $\mathbb{N}_+=\{1,2,3,\dots\}$;
    \item an integer $m\in\mathbb{N}_+$, called the \emph{decay parameter};
    \item a commutative ring $R$ with unity, in which all coefficients are defined;
    \item the polynomial ring $R[x]$ in one variable.
\end{itemize}

For each $n\in\mathbb{N}$, define
\[
w_m(n)=\left\lfloor \frac{n}{m} \right\rfloor.
\]

\begin{definition}[Triangular support of order $m$]\label{def:suporte-triangular}
Let $m\in\mathbb N_+$. A doubly indexed family
\[
\lambda:\mathbb{N}\times\mathbb{Z}\to R,\qquad \lambda=\{\lambda(n,k)\}_{n\ge0,\,k\in\mathbb{Z}},
\]
is said to have \emph{triangular support (of order $m$)} if, for all $n\in\mathbb N$ and $k\in\mathbb Z$,
\[
\lambda(n,k)=0\quad\text{whenever}\quad k<0\ \text{or}\ k>w_m(n),
\]
where $w_m(n)=\big\lfloor n/m\big\rfloor$.
If a family $\lambda$ has triangular support of order $m$, write $\lambda\in\Delta_m$.
\end{definition}

Whenever doubly indexed families $\lambda_1,\lambda_3:\mathbb{N}\times\mathbb{Z}\longrightarrow R$ appear,
it is assumed, unless otherwise stated, that $\lambda_j$ has triangular support of order $m$ for $j \in \{1,3\}$.

Throughout, polynomial families in $R[x]$ whose coefficients obey this triangular pattern are considered.
Accordingly, let
$\mathcal{F} = \{f_n(x)\}_{n \in \mathbb{N}} \subset R[x]$ be a family of polynomials whose exponents of $x$ belong to the set
$\{\,n - km : 0 \leq k \leq w_{m}(n)\,\}$. In this context, representations of the form
\[
f_n(x) = \sum_{b=0}^{w_{m}(n)} \lambda_1(n,b)\, x^{n - mb},
\]
are admitted, where $\lambda_1=\{\lambda_1(n,b)\}$ has triangular support of order $m$
(that is, $\lambda_1\in\Delta_m$). Moreover, it is assumed that each $f_n$ has degree exactly $n$,
so that the leading coefficient satisfies
\[
\lambda_1(n,0)\neq 0,\qquad n\ge0.
\]
The family $\lambda_1$ is called the \emph{direct kernel} associated with $\mathcal{F}$
(or simply the \emph{kernel}, when there is no ambiguity).
In sections where recursive formulas involving divisions by $\lambda_1(n,0)$ are used,
it is tacitly assumed that
\[
\lambda_1(n,0)\in R^{\times}\quad\text{for all }n\ge0,
\]
that is, the boundary coefficients are units in $R$.

\subsection{Reindexation principles}\label{subsec:reindex}

The Iverson bracket is used, denoted by $\Iverson{P}$, which equals $1$ if the proposition $P$ is true and $0$ otherwise.

\begin{proposition}[Convolution-type reindexation]\label{prop:conv-reindex}
Let $u\in\mathbb{N}$ and let $\{a_{j,k}\}$ be an arbitrary family. Then the following identity holds
(reindexed from \cite[Eq.~(2.32)]{GKP1994}):
\[
\sum_{k=0}^{u}\ \sum_{b=0}^{k} a_{b,k}
\;=\;
\sum_{b=0}^{u}\ \sum_{c=0}^{u-b} a_{\,b,\,b+c}.
\]
\end{proposition}

\begin{proof}
\[
\begin{aligned}
S
&= \sum_{k,b} a_{b,k}\,\Iverson{0\le b\le k\le u}
\\[2mm]
&= \sum_{k,b} a_{b,k}\,\Iverson{0\le b\le u}\,\Iverson{0\le k-b\le u-b}
\\[2mm]
&= \sum_{b,c} a_{\,b,\,b+c}\,\Iverson{0\le b\le u}\,\Iverson{0\le c\le u-b}
\quad\text{($c=k-b\ \Rightarrow\ k=b+c$).}
\end{aligned}
\]
Since both expressions define the same $S$, the identity follows.
\end{proof}

\subsection{Inversion under triangular support}\label{subsec:inv-triangular}

Let $\mathcal{F}=\{f_n(x)\}_{n\in\mathbb{N}}$ be defined by
\begin{equation}\label{eq:def-familia}
    f_n(x)=\sum_{b=0}^{w_m(n)} \lambda_1(n,b)\,x^{n-mb},
\end{equation}
where $m\in\mathbb{N}_+$ is the decay parameter associated with $w_m$.
Assume that, for all $n\in\mathbb{N}$,
\[
\lambda_1(n,0)\,\lambda_3(n,0)=1.
\]

\begin{theorem}[Inversion between $\{x^n\}$ and $\{f_n\}$]\label{thm:inversao-direta-base}
Under the above assumptions, the following are equivalent:
\begin{enumerate}[label=(\roman*)]
    \item for all $1\le k\le w_m(n)$,
    \begin{equation}\label{eq:ortogonalidade-discreta}
        \sum_{b=0}^{k}\lambda_3(n,b)\,\lambda_1(n-mb,k-b)=0;
    \end{equation}
    \item the inversion identity
    \begin{equation}\label{eq:inversao-geral}
        x^n=\sum_{b=0}^{w_m(n)} \lambda_{3}(n,b)\,f_{n-mb}(x).
    \end{equation}
\end{enumerate}
\end{theorem}

\begin{proof}
For each $b\ge 0$,
\[
\lambda_3(n,b)f_{n-mb}(x)
=\lambda_3(n,b)\sum_{c=0}^{w_m(n-mb)}\lambda_1(n-mb,c)\,x^{n-mb-mc},
\]
and summing over $b=0,\dots,w_m(n)$ yields
\[
\sum_{b=0}^{w_m(n)}\lambda_{3}(n,b)f_{n-mb}(x)
=\sum_{b=0}^{w_m(n)}\sum_{c=0}^{w_m(n-mb)}\lambda_{3}(n,b)\,\lambda_{1}(n-mb,c)\,x^{n-mb-mc}.
\]
Using the identity $w_m(n-mb)=w_m(n)-b$ and applying Proposition~\ref{prop:conv-reindex}, this becomes
\begin{equation}\label{eq:soma-reindexada}
\begin{split}
\sum_{b=0}^{w_m(n)}\lambda_{3}(n,b)f_{n-mb}(x)
= \sum_{k=0}^{w_m(n)} \,x^{n-mk} \sum_{b=0}^{k}\lambda_{3}(n,b)\,\lambda_{1}(n-mb,k-b)\\
= \lambda_{3}(n,0)\lambda_{1}(n,0)\,x^{n}
  + \sum_{k=1}^{w_m(n)} \,x^{n-mk} \sum_{b=0}^{k}\lambda_{3}(n,b)\,\lambda_{1}(n-mb,k-b).
\end{split}
\end{equation}

\emph{(i) $\Rightarrow$ (ii).}
The coefficient of $x^{n}$ in \eqref{eq:soma-reindexada} (the case $k=0$) is
$\lambda_3(n,0)\lambda_1(n,0)=1$ and, if \eqref{eq:ortogonalidade-discreta} holds,
the coefficients for $1\le k\le w_m(n)$ vanish. Therefore, from \eqref{eq:soma-reindexada} it follows that
\[
x^n=\sum_{b=0}^{w_m(n)} \lambda_{3}(n,b)\,f_{n-mb}(x),
\]
that is, \eqref{eq:inversao-geral}.

\medskip
\emph{(ii) $\Rightarrow$ (i).}
If \eqref{eq:inversao-geral} holds, then comparing with \eqref{eq:soma-reindexada} yields
\[
\sum_{k=1}^{w_m(n)}x^{n-mk}\left(\sum_{b=0}^{k}\lambda_{3}(n,b)\,\lambda_{1}(n-mb,k-b)\right)=0.
\]
Since the monomials $\{x^{n-mk}\}_{k\ge 1}$ are linearly independent, for all
$1\le k\le w_m(n)$,
\[
\sum_{b=0}^{k}\lambda_{3}(n,b)\,\lambda_{1}(n-mb,k-b)=0,
\]
that is, \eqref{eq:ortogonalidade-discreta}.
\end{proof}

\begin{example}[Computational example with random coefficients]\label{ex:m1n6-visual-vazio}
The parameters $m=1$, $n=6$, and random coefficients in $\{-7,\cdots,-1,1,\cdots,7\}$ are used.
Since $w_{1}(6)=6$, the orthogonality recurrence
\[
\lambda_{3}(n,k)
=
-\sum_{b=0}^{k-1}\frac{\lambda_{3}(n,b)\,\lambda_{1}(n-mb,k-b)}{\lambda_{1}(n-mk,0)}
\]
(computed using \texttt{Python}) yields
\[
(\lambda_3(6,k))_{0\le k \le 6} =\bigl(1,\ 3,\ -1,\ 4,\ 26,\ 25,\ -\tfrac{40}{7}\bigr).
\]

\begin{table}[ht]
\centering
\caption{The family $\{f_u\}_{u=0}^{6}$ with random coefficients.}
\label{tab:familia-m1n6-pol-vazio}
\renewcommand{\arraystretch}{1.2}
\setlength{\tabcolsep}{6pt}
\begin{tabular}{c|ccccccc}
\hline
$n$ & $x^6$ & $x^5$ & $x^4$ & $x^3$ & $x^2$ & $x$ & $1$\\
\hline
$0$ &  &  &  &  &  &  & $\phantom{+}7$\\
$1$ &  &  &  &  &  & $-6x$ & $\phantom{+}6$\\
$2$ &  &  &  &  & $\phantom{+}x^2$ & $\phantom{+}5x$ & $-5$\\
$3$ &  &  &  & $\phantom{+}6x^3$ & $-6x^2$ & $\phantom{+}2x$ & $\phantom{+}3$\\
$4$ &  &  & $-7x^4$ & $-3x^3$ & $\phantom{+}6x^2$ & $-4x$ & $-4$\\
$5$ &  & $-2x^5$ & $-3x^4$ & $-7x^3$ & $\phantom{+}x^2$ & $\phantom{+}x$ & $\phantom{+}1$\\
$6$ & $\phantom{+}x^6$ & $\phantom{+}6x^5$ & $\phantom{+}2x^4$ & $-6x^3$ & $\phantom{+}x^2$ & $\phantom{+}5x$ & $\phantom{+}1$\\
\hline
\end{tabular}
\end{table}

\begin{table}[ht]
\centering
\caption{Visualization of $x^n=\sum_{b=0}^{w_m(n)}\lambda_3(n,b)\,f_{n-mb}(x)$.}
\label{tab:inversao-m1n6-pol-vazio}
\renewcommand{\arraystretch}{1.2}
\setlength{\tabcolsep}{6pt}
\begin{tabular}{l|ccccccc}
\hline
Term & $x^6$ & $x^5$ & $x^4$ & $x^3$ & $x^2$ & $x$ & $1$\\
\hline
$\lambda_3(6,0)f_{6}(x)$ & $\phantom{+}x^6$ & $\phantom{+}6x^5$ & $\phantom{+}2x^4$ & $-6x^3$ & $\phantom{+}x^2$ & $\phantom{+}5x$ & $\phantom{+}1$\\
$\lambda_3(6,1)f_{5}(x)$ &  & $-6x^5$ & $-9x^4$ & $-21x^3$ & $\phantom{+}3x^2$ & $\phantom{+}3x$ & $\phantom{+}3$\\
$\lambda_3(6,2)f_{4}(x)$ &  &  & $\phantom{+}7x^4$ & $\phantom{+}3x^3$ & $-6x^2$ & $\phantom{+}4x$ & $\phantom{+}4$\\
$\lambda_3(6,3)f_{3}(x)$ &  &  &  & $\phantom{+}24x^3$ & $-24x^2$ & $\phantom{+}8x$ & $\phantom{+}12$\\
$\lambda_3(6,4)f_{2}(x)$ &  &  &  &  & $\phantom{+}26x^2$ & $\phantom{+}130x$ & $-130$\\
$\lambda_3(6,5)f_{1}(x)$ &  &  &  &  &  & $-150x$ & $\phantom{+}150$\\
$\lambda_3(6,6)f_{0}(x)$ &  &  &  &  &  &  & $-40$\\
\hline
Sum & $\phantom{+}x^6$ &  &  &  &  &  & \\
\hline
\end{tabular}
\end{table}
\end{example}

\begin{definition}[Inverse kernel]\label{def:nucleo-inverso}
Let $m\in\mathbb{N}_+$ and let
$\mathcal{F}=\{f_n(x)\}_{n\in\mathbb{N}}\subset R[x]$
be a polynomial family with triangular support of order $m$, admitting a
representation
\[
f_n(x)=\sum_{b=0}^{w_m(n)}\lambda_1(n,b)\,x^{n-mb},
\qquad \lambda_1\in\Delta_m,
\]
whose direct kernel is $\lambda_1$.
A doubly indexed family $\lambda_3\in\Delta_m$ is called the
\emph{inverse kernel} associated with $\mathcal{F}$ if, for all
$n\in\mathbb{N}$,
\begin{equation}\label{eq:def-nucleo-inverso}
    x^n \;=\; \sum_{k=0}^{w_m(n)} \lambda_3(n,k)\,f_{n-mk}(x).
\end{equation}
\end{definition}

Whenever such a family $\lambda_3$ exists, $\mathcal{F}$ is said to admit an
\emph{inverse kernel of order $m$}. Equivalently, $\lambda_1$ and $\lambda_3$
implement a triangular change of basis between $\{x^n\}$ and $\{f_n\}$.

\begin{theorem}[Uniqueness of inversion coefficients]\label{thm:unicidade-inversao-indutiva}
Fix $m\in\mathbb{N}_+$ and let $w_m(n)=\lfloor n/m\rfloor$. 
Let $\mathcal{G}=\{g_n(x)\}_{n\ge0}$ be given by
\[
g_n(x)=\sum_{c=0}^{w_m(n)} \mu_1(n,c)\,x^{\,n-mc},
\qquad
\mu_1(n,0)\in R^\times\ \text{ for all }n.
\]
If two families $\{\mu_3^{(1)}(n,k)\}$ and $\{\mu_3^{(2)}(n,k)\}$ satisfy, for all $n$,
\[
x^n=\sum_{k=0}^{w_m(n)} \mu_3^{(1)}(n,k)\,g_{n-mk}(x)
\;=\;
\sum_{k=0}^{w_m(n)} \mu_3^{(2)}(n,k)\,g_{n-mk}(x),
\]
then $\mu_3^{(1)}(n,k)=\mu_3^{(2)}(n,k)$ for all admissible indices. In particular, the inversion coefficients are unique whenever they exist.
\end{theorem}

\begin{proof}
Fix $n$ and consider
\[
0
=\sum_{k=0}^{w_m(n)} \bigl(\mu_3^{(1)}(n,k)-\mu_3^{(2)}(n,k)\bigr)\,g_{n-mk}(x).
\]
By the expansion of $g_n$,
\[
g_n(x)=\mu_1(n,0)\,x^n+\text{(degrees }\le n-m),
\qquad
g_{n-mk}(x)\ \text{has degree }\le n-m\ \ (k\ge1).
\]
Therefore, the coefficient of $x^n$ in the identity above is
\[
\bigl(\mu_3^{(1)}(n,0)-\mu_3^{(2)}(n,0)\bigr)\,\mu_1(n,0)=0,
\]
and since $\mu_1(n,0)\in R^\times$, it follows that $\mu_3^{(1)}(n,0)=\mu_3^{(2)}(n,0)$.

For $t=1,\dots,w_m(n)$, assume inductively that
$\mu_3^{(1)}(n,k)=\mu_3^{(2)}(n,k)$ for all $k<t$. Cancelling the terms with $k<t$ in the initial identity yields
\[
0
=\sum_{k=t}^{w_m(n)} \bigl(\mu_3^{(1)}(n,k)-\mu_3^{(2)}(n,k)\bigr)\,g_{n-mk}(x).
\]
Now consider the coefficient of $x^{\,n-tm}$. One has
\[
g_{n-mt}(x)=\mu_1(n-mt,0)\,x^{\,n-tm}+\text{(degrees }\le n-(t+1)m),
\]
while every $g_{n-mk}(x)$ with $k>t$ has degree $\le n-(t+1)m$ and does not contribute. Thus,
\[
\bigl(\mu_3^{(1)}(n,t)-\mu_3^{(2)}(n,t)\bigr)\,\mu_1(n-mt,0)=0,
\]
and since $\mu_1(n-mt,0)\in R^\times$, it follows that $\mu_3^{(1)}(n,t)=\mu_3^{(2)}(n,t)$.

Induction on $t$ closes for all $0\le t\le w_m(n)$, and since $n$ is arbitrary, uniqueness holds at all levels.
\end{proof}

\begin{theorem}[Decomposition into residue classes]\label{thm:decomp-classes}
Fix $m\in\mathbb N_+$ and write $w_m(n)=\lfloor n/m\rfloor$.
Assume that $\{f_n(x)\}_{n\ge0}$ admits a direct kernel $\lambda_1$ of order $m$, that is,
\[
f_n(x)
=
\sum_{b=0}^{w_m(n)} \lambda_1(n,b)\,x^{\,n-mb},
\qquad n\ge0.
\]
Under the hypotheses of Theorem~\ref{thm:inversao-direta-base}, let $\lambda_3$ be the corresponding inverse kernel,
so that \eqref{eq:ortogonalidade-discreta} and \eqref{eq:inversao-geral} hold.

For each residue $r\in\{0,1,\dots,m-1\}$, write $n=mk+r$ with $k\in\mathbb N$ and define
\[
f^{\langle r\rangle}_k(x)=f_{mk+r}(x),\qquad k\ge0.
\]
Define further, for $0\le t\le k$,
\[
\lambda_1^{\langle r\rangle}(k,t)
\;=\;
\lambda_1(mk+r,\,k-t),
\qquad
\lambda_3^{\langle r\rangle}(k,t)
\;=\;
\lambda_3(mk+r,\,t).
\]

Then, for each residue $r$, the following hold:
\begin{enumerate}[label=(\alph*)]
    \item (Classwise direct kernel) For all $k\ge0$,
  \begin{equation}\label{eq:classe-direta}
      f^{\langle r\rangle}_k(x)
      =
      \sum_{t=0}^{k}
      \lambda_1^{\langle r\rangle}(k,t)\,x^{\,r+mt}.
  \end{equation}
 \item (Classwise orthogonality and inversion) The following are equivalent:
\begin{enumerate}[label=(\roman*), leftmargin=*, nosep]
\item for $0\le r<m$ and $1\le j\le k$,
\[
\sum_{t=0}^{j} \lambda_3^{\langle r\rangle}(k,t)\,
\lambda_1^{\langle r\rangle}(k-t,\,k-j)=0;
\]
\item for $0\le r<m$ and $k\ge0$, the inversion identity
\[
x^{mk+r}=\sum_{t=0}^{k}\lambda_3^{\langle r\rangle}(k,t)\,
f^{\langle r\rangle}_{\,k-t}(x).
\]
\end{enumerate}
  Moreover,
\[
    \{\lambda_3(n,t)\}_{n\ge0,\,t\in\mathbb Z}
    =
    \bigcup_{r=0}^{m-1}
    \{\lambda_3^{\langle r\rangle}(k,t)\}_{k\ge0,\,t\in\mathbb Z}.
\]
\end{enumerate}
\end{theorem}

\begin{proof}
Fix $r\in\{0,\dots,m-1\}$ and $k\ge0$. Since $w_m(mk+r)=k$, the definition of $\{f_n(x)\}_{n\ge0}$ gives
\[
f_{mk+r}(x)
=
\sum_{b=0}^{k}
\lambda_1(mk+r,b)\,x^{\,mk+r-mb}.
\]
Writing $t=k-b$, one has $0\le t\le k$ and $mk+r-mb=r+mt$, so that
\[
f^{\langle r\rangle}_k(x)
=f_{mk+r}(x)=
\sum_{t=0}^{k}
\lambda_1(mk+r,\,k-t)\,x^{\,r+mt}
=\sum_{t=0}^{k}
\lambda_1^{\langle r\rangle}(k,\,t)\,x^{\,r+mt},
\]
which is precisely \eqref{eq:classe-direta} by the definition of
$\lambda_1^{\langle r\rangle}$. This proves item~(a).

For item~(b), take $n=mk+r$ in \eqref{eq:ortogonalidade-discreta}. Since $w_m(n)=k$, one obtains, for $1\le j\le k$,
\begin{align*}
0
&=\sum_{t=0}^{j}\lambda_3(mk+r,t)\,\lambda_1\bigl(mk+r-mt,\,j-t\bigr)\\
&=\sum_{t=0}^{j}
\lambda_3(mk+r,t)\,\lambda_1\bigl(m(k-t)+r,\,j-t\bigr)\\
&=\sum_{t=0}^{j} \lambda_3^{\langle r\rangle}(k,t)\,
\lambda_1^{\langle r\rangle}(k-t,\,k-j),
\end{align*}
where, in the last equality, the definitions of
$\lambda_3^{\langle r\rangle}$ and $\lambda_1^{\langle r\rangle}$ were used. This establishes subitem~\textup{(i)}.

Taking again $n=mk+r$ in \eqref{eq:inversao-geral}, one obtains
\begin{align*}
x^{mk+r}
&=\sum_{t=0}^{k}\lambda_3(mk+r,t)\,f_{mk+r-mt}(x)\\
&=\sum_{t=0}^{k}\lambda_3(mk+r,t)\,f_{m(k-t)+r}(x)\\    
&=\sum_{t=0}^{k}\lambda_3^{\langle r\rangle}(k,t)\,
f^{\langle r\rangle}_{\,k-t}(x),
\end{align*}
which is precisely the inversion identity in subitem~\textup{(ii)}.

Since, under the hypotheses of the theorem, \eqref{eq:ortogonalidade-discreta} and
\eqref{eq:inversao-geral} are equivalent, and the identities in
subitems~\textup{(i)} and~\textup{(ii)} are obtained from these global equations
solely by bijective reindexations of the indices, it follows that
\textup{(i)} and \textup{(ii)} are also equivalent, as stated in item~(b).

Finally, each $n\ge0$ admits a unique decomposition $n=mk+r$ with
$0\le r<m$, and by definition
\[
\lambda_3^{\langle r\rangle}(k,t)
=
\lambda_3(mk+r,t).
\]
Thus, every pair $(n,t)$ can be written uniquely as
$(n,t)=(mk+r,t)$ for some $0\le r<m$ and $k\ge0$, and one obtains
\[
    \{\lambda_3(n,t)\}_{n\ge0,\,t\in\mathbb Z}
    =
    \bigcup_{r=0}^{m-1}
    \{\lambda_3^{\langle r\rangle}(k,t)\}_{k\ge0,\,t\in\mathbb Z},
\]
which completes the proof.
\end{proof}

\begin{corollary}[Composition of inversions]\label{thm:composicao}
Let $m\in\mathbb{N}_+$. Consider the families
\begin{equation}\label{eq:familias-fg}
    f_n(x)=\sum_{b=0}^{w_m(n)} \lambda_1(n,b)\,x^{\,n-mb},
    \qquad
    g_n(x)=\sum_{c=0}^{w_m(n)} \mu_1(n,c)\,x^{\,n-mc}.
\end{equation}
Assume that triangular inverse kernels $\lambda_3$ and $\mu_3$ associated with $\{f_n\}$ and $\{g_n\}$ exist, that is, for all $n\in\mathbb{N}$,
\begin{equation}\label{eq:inversoes-fg}
\begin{aligned}
    x^n &= \sum_{b=0}^{w_m(n)} \lambda_{3}(n,b)\, f_{n-mb}(x),\\
    x^n &= \sum_{c=0}^{w_m(n)} \mu_{3}(n,c)\, g_{n-mc}(x).
\end{aligned}
\end{equation}
Then, for all $n\in\mathbb{N}$,
\begin{equation}\label{eq:relacao-fg1}
    f_n(x)
    = \sum_{k=0}^{w_m(n)} g_{n-mk}(x)\,z(n,k),
\end{equation}
with
\begin{equation}\label{eq:relacao-fg2}
    z(n,k)
    = \sum_{b=0}^{k}\lambda_1(n,b)\,\mu_{3}(n-mb,k-b).
\end{equation}
\end{corollary}

\begin{proof}
From the definition of $f_n$,
\[
f_n(x)=\sum_{b=0}^{w_m(n)} \lambda_1(n,b)\,x^{\,n-mb}.
\]
For each $b$, apply the inversion identity for $g$ in \eqref{eq:inversoes-fg} to the shifted index $n-mb$:
\[
x^{\,n-mb}
=\sum_{c=0}^{w_m(n-mb)}
\mu_3(n-mb,c)\,
g_{n-m(b+c)}(x).
\]
Substituting into $f_n(x)$ and using $w_m(n-mb)=w_m(n)-b$ yields
\[
f_n(x)
=\sum_{b=0}^{w_m(n)} \lambda_1(n,b)
\sum_{c=0}^{w_m(n)-b}
\mu_3(n-mb,c)\,
g_{n-m(b+c)}(x).
\]
Reindexing by $k=b+c$ and applying Proposition~\ref{prop:conv-reindex}, one obtains
\[
f_n(x)
=\sum_{k=0}^{w_m(n)} g_{n-mk}(x)\,
\sum_{b=0}^{k}
\lambda_1(n,b)\,\mu_3(n-mb,k-b),
\]
that is, \eqref{eq:relacao-fg1}–\eqref{eq:relacao-fg2}.
\end{proof}

\begin{theorem}[Composition of inversions with distinct orders]\label{thm:composicao-m1-m2}
Fix $m_1,m_2\in\mathbb N_+$ and, for $n\ge0$, write $w_{m_i}(n)=\lfloor n/m_i\rfloor$ $(i=1,2)$.
Consider
\[
f_n(x)=\sum_{b=0}^{w_{m_1}(n)} \lambda_1(n,b)\,x^{\,n-m_1b},
\qquad
g_n(x)=\sum_{c=0}^{w_{m_2}(n)} \mu_1(n,c)\,x^{\,n-m_2c},
\]
with direct kernels $\lambda_1\in\Delta_{m_1}$ and $\mu_1\in\Delta_{m_2}$, respectively.
Assume that triangular inverse kernels $\lambda_3$ and $\mu_3$ exist such that, for all $n\ge0$,
\begin{equation}\label{eq:inv-m1m2}
x^n=\sum_{b=0}^{w_{m_1}(n)} \lambda_3(n,b)\,f_{n-m_1b}(x),
\qquad
x^n=\sum_{c=0}^{w_{m_2}(n)} \mu_3(n,c)\,g_{n-m_2c}(x).
\end{equation}
Then, for all $n\in\mathbb N$, the following change-of-basis expansion holds:
\begin{equation}\label{eq:fg-m1m2}
f_n(x)=\sum_{r=0}^{n} g_r(x)\,Z(n;r),
\end{equation}
where
\begin{equation}\label{eq:Z-m1m2}
Z(n;r)\;=\;\sum_{\substack{
0\le b\le w_{m_1}(n)\\[1pt]
r\le n-m_1b\\[1pt]
n-m_1b-r\equiv 0\ (\mathrm{mod}\ m_2)
}}
\lambda_1(n,b)\;
\mu_3\!\Bigl(n-m_1b,\ \frac{n-m_1b-r}{m_2}\Bigr).
\end{equation}
In the particular case $m_1=m_2=m$, writing $r=n-mk$ with $0\le k\le w_m(n)$, one obtains
\[
Z(n;r)=\sum_{b=0}^{k} \lambda_1(n,b)\,\mu_3(n-mb,\,k-b),
\]
in agreement with \eqref{eq:relacao-fg2}.
\end{theorem}

\begin{proof}
From the definition of $f_n$,
\[
f_n(x)=\sum_{b=0}^{w_{m_1}(n)} \lambda_1(n,b)\,x^{\,n-m_1b}.
\]
For each $b$, apply the inversion identity for $g$ in \eqref{eq:inv-m1m2} to the shifted index $n-m_1b$:
\[
x^{\,n-m_1b}
=\sum_{c=0}^{w_{m_2}(n-m_1b)} \mu_3(n-m_1b,c)\,g_{n-m_1b-m_2c}(x).
\]
Substituting into $f_n(x)$ and interchanging the sums yields
\[
f_n(x)=\sum_{b=0}^{w_{m_1}(n)} \lambda_1(n,b)
\sum_{c=0}^{w_{m_2}(n-m_1b)} \mu_3(n-m_1b,c)\,g_{n-m_1b-m_2c}(x).
\]

Setting
\[
r=n-m_1b-m_2c,
\]
the pairs $(b,c)$ contributing to a given $r$ are exactly those such that
\[
0\le b\le w_{m_1}(n),\quad r\le n-m_1b,\quad n-m_1b-r=m_2c\ge0,
\]
that is, $n-m_1b-r$ is a multiple of $m_2$ and $c=(n-m_1b-r)/m_2\in\mathbb N$.
In particular, writing $g=\gcd(m_1,m_2)$, the Diophantine equation
\[
r=n-m_1b-m_2c
\]
admits an integer solution $(b,c)$ if and only if $r\equiv n\pmod g$; hence, for $r\not\equiv n\pmod g$ the inner sum in \eqref{eq:Z-m1m2} is empty and therefore $Z(n;r)=0$.

Regrouping by $r$ yields
\[
f_n(x)=\sum_{r=0}^{n} g_r(x)\,
\sum_{\substack{
0\le b\le w_{m_1}(n)\\
r\le n-m_1b\\
n-m_1b-r\equiv 0\ (\mathrm{mod}\ m_2)
}}
\lambda_1(n,b)\;
\mu_3\!\Bigl(n-m_1b,\ \frac{n-m_1b-r}{m_2}\Bigr),
\]
which is precisely \eqref{eq:fg-m1m2}–\eqref{eq:Z-m1m2}. In the case $m_1=m_2=m$, taking $r=n-mk$ with
$0\le k\le w_m(n)$ immediately yields the announced triangular expression.
\end{proof}

\begin{proposition}[Reindexation by residue classes]\label{prop:reindex-classes}
Let $m\in\mathbb{N}_+$, $n\in\mathbb{N}$, and let $\{a_{j,k}\}$ be a family. Then
\[
\sum_{k=0}^{n}\ \sum_{j=0}^{k} a_{j,k}\,\Iverson{k\equiv j \!\!\!\pmod m}
\;=\;
\sum_{k=0}^{n}\ \sum_{t=0}^{\lfloor k/m\rfloor} a_{\,k-mt,\,k}
\;=\;
\sum_{j=0}^{n}\ \sum_{t=0}^{\left\lfloor\frac{n-j}{m}\right\rfloor} a_{\,j,\,j+mt}.
\]
\end{proposition}

\begin{proof}
\[
\begin{aligned}
S
&= \sum_{j,k} a_{j,k}\,\Iverson{0\le k\le n}\,\Iverson{0\le j\le k}\,\Iverson{k\equiv j \!\!\!\pmod m}
\\[2mm]
&= \sum_{k,t} a_{\,k-mt,\,k}\,\Iverson{0\le k\le n}\,\Iverson{0\le t\le \left\lfloor\frac{k}{m}\right\rfloor}
\\[2mm]
&= \sum_{j,t} a_{\,j,\,j+mt}\,\Iverson{0\le j\le n}\,\Iverson{0\le t\le \left\lfloor\frac{n-j}{m}\right\rfloor}.
\end{aligned}
\]
Since the three expressions define the same $S$, the equalities follow.
\end{proof}

\begin{theorem}[Inversion by residue classes]\label{thm:inversao-classes}
Fix $m\in\mathbb N_+$. Let $\mathcal F=\{f_{n}(x)\}_{n\ge0}$ be a polynomial family of order $m$ such that, for all $n\in\mathbb N$,
\[
x^{n}=\sum_{b=0}^{w_m(n)} \lambda_{3}(n,b)\,f_{n-mb}(x),
\]
where $\lambda_3\in\Delta_m$ is a triangular inverse kernel associated with $\mathcal F$.
For a polynomial $p(x)=\sum_{a=0}^{N} c_a\,x^a$ (with $N\in\mathbb N$), the classwise coordinates
\[
C(r)\;=\;\sum_{t=0}^{\left\lfloor\frac{N-r}{m}\right\rfloor} c_{\,r+mt}\,\lambda_{3}(r+mt,t)\qquad(0\le r\le N)
\]
hold, and the expansion of $p$ in the basis $\mathcal F$ is
\begin{equation}\label{eq:escada-forma-final}
p(x)
=\sum_{r=0}^{N} C(r)\,f_{r}(x)
=\sum_{r=0}^{N} f_{r}(x)\,
\sum_{t=0}^{\left\lfloor\frac{N-r}{m}\right\rfloor}
c_{\,r+mt}\,\lambda_{3}(r+mt,t).
\end{equation}
\end{theorem}

\begin{proof}
Write $p(x)=\sum_{a=0}^{N} c_a x^a$ and, for each $a$, substitute the inversion
\[
x^a=\sum_{b=0}^{w_m(a)}\lambda_3(a,b)\,f_{a-mb}(x).
\]
One obtains
\[
p(x)=\sum_{a=0}^{N} c_a \sum_{b=0}^{w_m(a)} \lambda_3(a,b)\,f_{a-mb}(x).
\]
Reindex by the residue class $r=a-mb$ (that is, $a=r+mb$) and use
Proposition~\ref{prop:reindex-classes} to reorganize the sums:
\[
p(x)=\sum_{r=0}^{N} f_{r}(x)\,
\sum_{t=0}^{\left\lfloor\frac{N-r}{m}\right\rfloor} c_{\,r+mt}\,\lambda_{3}(r+mt,t),
\]
which is exactly \eqref{eq:escada-forma-final}.
\end{proof}

\begin{example}[Preprocessing for Clenshaw-type schemes]
Under the hypotheses of Theorem~\ref{thm:inversao-classes}, consider a
polynomial $p(x)=\sum_{a=0}^{N} c_a x^a$. By the inversion-by-classes formula, one may rewrite
\[
p(x)=\sum_{r=0}^{N} C(r)\,f_{r}(x),
\qquad
C(r)=\sum_{t=0}^{\left\lfloor\frac{N-r}{m}\right\rfloor}
c_{\,r+mt}\,\lambda_{3}(r+mt,t).
\]
In other words, the inverse kernel $\lambda_3$ provides a
preprocessing of the coefficients $\{c_a\}$ into new coefficients
$\{C(r)\}$ adapted to the basis $\{f_r\}$. If, in addition, the family
$\{f_{r}(x)\}_{r\ge0}$ satisfies a three-term recurrence in $r$,
then the sum $\sum_{r=0}^{N} C(r)\,f_{r}(x)$ can be evaluated by a
Clenshaw-type scheme adapted to the basis $\{f_r\}$, using the coefficients
$C(r)$.
\end{example}

\begin{theorem}[Determinantal formula for the triangular inverse kernel]\label{thm4}
Let $m\in\mathbb{N}_+$ and $n\in\mathbb{N}$, and write $w_m(n)=\lfloor n/m\rfloor$.
Consider a direct kernel
\[
\lambda_1\in\Delta_m
\]
(with entries in $R$). Suppose that, for this fixed $n$, there exist coefficients
$\lambda_3(n,k)\in R$ satisfying:
\begin{enumerate}[label={\normalfont(\alph*)}]
    \item the discrete orthogonality condition
    \[
      \sum_{b=0}^{t}\lambda_{3}(n,b)\,\lambda_{1}(n-mb,t-b)=0
      \quad\text{for all } 1\le t\le w_m(n);
    \]
    \item the normalization
    \[
      \lambda_1(n,0)\,\lambda_3(n,0)=1;
    \]
    \item $\lambda_{1}(n-im,0)\in R^{\times}$ for all $1\le i\le w_m(n)$, that is,
    all diagonal coefficients $\lambda_1(n-im,0)$ are units in $R$.
\end{enumerate}
Then, for each $1\le k\le w_m(n)$, the coefficients $\lambda_3(n,k)$ are given by
\[
\lambda_{3}(n,k)=(-1)^k\,
\frac{\det A_k}{\displaystyle\prod_{i=0}^{k}\lambda_{1}(n-im,0)},
\]
where $A_k$ is the $k\times k$ lower Hessenberg matrix
\[
A_k=\begin{bmatrix}
\lambda_{1}(n,1) & \lambda_{1}(n-m,0) & 0 & \cdots & 0\\
\lambda_{1}(n,2) & \lambda_{1}(n-m,1) & \lambda_{1}(n-2m,0) & \cdots & 0\\
\lambda_{1}(n,3) & \lambda_{1}(n-m,2) & \lambda_{1}(n-2m,1) & \ddots & \vdots\\
\vdots & \vdots & \vdots & \ddots & \lambda_{1}(n-(k-1)m,0)\\
\lambda_{1}(n,k) & \lambda_{1}(n-m,k-1) & \lambda_{1}(n-2m,k-2) & \cdots & \lambda_{1}(n-(k-1)m,1)
\end{bmatrix}.
\]
\end{theorem}

\begin{proof}
From the discrete orthogonality condition for $t=k$,
\[
\sum_{b=0}^{k}\lambda_{3}(n,b)\,\lambda_{1}(n-mb,k-b)=0.
\]
Isolating the term $b=k$ yields
\[
\lambda_3(n,k)\,\lambda_1(n-mk,0)
=-\sum_{b=0}^{k-1}\lambda_3(n,b)\,\lambda_1(n-mb,k-b),
\]
that is,
\[
\lambda_3(n,k)=-\frac{\displaystyle\sum_{b=0}^{k-1}\lambda_3(n,b)\,\lambda_1(n-mb,k-b)}{\lambda_1(n-mk,0)},
\qquad 1\le k\le w_m(n).
\]

Define $P_{(n,0)}=1$ and, for $k\ge 1$,
\[
P_{(n,k)}=-\frac{\displaystyle\sum_{b=0}^{k-1}P_{(n,b)}\,\lambda_1(n-mb,k-b)}{\lambda_1(n-mk,0)}.
\]
It is shown by induction on $k$ that
\[
\lambda_3(n,k)=\lambda_3(n,0)\,P_{(n,k)}
\quad\text{for } 0\le k\le w_m(n).
\]
The base case $k=0$ is immediate. For the inductive step, assume the identity holds for
$b<k$; substituting into the recurrence for $\lambda_3$ gives
\[
\lambda_3(n,k)
= -\frac{\displaystyle\lambda_3(n,0)\sum_{b=0}^{k-1}P_{(n,b)}\,\lambda_1(n-mb,k-b)}{\lambda_1(n-mk,0)}
= \lambda_3(n,0)\,P_{(n,k)}.
\]

The recurrence for $P_{(n,k)}$ is now organized into a lower triangular linear system. For
$1\le t\le k$, multiplying the definition of $P_{(n,t)}$ by $\lambda_1(n-mt,0)$ and isolating the term with $b=0$ (using $P_{(n,0)}=1$) yields
\[
\lambda_{1}(n-mt,0)\,P_{(n,t)}
+\sum_{i=1}^{t-1}\lambda_{1}\bigl(n-im,t-i\bigr)\,P_{(n,i)}
=-\,\lambda_{1}(n,t).
\]
In matrix notation, this is
\[
T_k\,\mathbf{P}=\mathbf{c},
\]
where
\begin{align*}
    T_k&=
\begin{bmatrix}
\lambda_{1}(n-m,0) & 0 & \cdots & 0\\
\lambda_{1}(n-m,1) & \lambda_{1}(n-2m,0) & \cdots & 0\\
\vdots & \vdots & \ddots & \vdots\\
\lambda_{1}(n-m,k-1) & \lambda_{1}(n-2m,k-2) & \cdots & \lambda_{1}(n-km,0)
\end{bmatrix},\\
\mathbf{P}&=\begin{bmatrix}P_{(n,1)}\\ \vdots\\ P_{(n,k)}\end{bmatrix},\\
\mathbf{c}&=\begin{bmatrix}-\lambda_{1}(n,1)\\ \vdots\\ -\lambda_{1}(n,k)\end{bmatrix}.
\end{align*}

The matrix $T_k$ is lower triangular, with
\[
\det T_k=\prod_{i=1}^{k}\lambda_{1}(n-im,0),
\]
which is a unit in $R$ by hypothesis {\normalfont(c)}. Thus, the system has a unique solution and Cramer's rule applies:
\[
P_{(n,k)}=\frac{\det T_k^{(k)}}{\det T_k},
\]
where $T_k^{(k)}$ is the matrix obtained from $T_k$ by replacing the $k$-th column by $\mathbf{c}$.

To place the vector of constants in the first column, move this column of $T_k^{(k)}$ to the first position: this performs $k-1$ adjacent column swaps, introducing the factor $(-1)^{k-1}$. The resulting matrix has first column
\[
\begin{bmatrix}-\lambda_{1}(n,1)\\ \vdots\\ -\lambda_{1}(n,k)\end{bmatrix}
\]
and the remaining columns exactly as in $A_k$. Factoring $-1$ from the first column yields an additional factor $(-1)$ in the determinant and produces $A_k$. In summary,
\[
\det T_k^{(k)} = (-1)^{k-1}\cdot(-1)\,\det A_k = (-1)^k\,\det A_k,
\]
where $A_k$ is precisely the lower Hessenberg matrix displayed in the statement.
It follows that
\[
P_{(n,k)}=(-1)^k\,\frac{\det A_k}{\displaystyle\prod_{i=1}^{k}\lambda_{1}(n-im,0)}.
\]

Finally,
\[
\lambda_3(n,k)
=\lambda_3(n,0)\,P_{(n,k)}
=(-1)^k\,\lambda_3(n,0)\,
\frac{\det A_k}{\displaystyle\prod_{i=1}^{k}\lambda_{1}(n-im,0)}.
\]
By the normalization $\lambda_1(n,0)\,\lambda_3(n,0)=1$, one has
\(
\lambda_3(n,0)=\lambda_1(n,0)^{-1}
\),
so that
\[
\lambda_3(n,k)
=(-1)^k\,
\frac{\det A_k}{\displaystyle\lambda_1(n,0)\prod_{i=1}^{k}\lambda_{1}(n-im,0)}
= (-1)^k\,
\frac{\det A_k}{\displaystyle\prod_{i=0}^{k}\lambda_{1}(n-im,0)}.
\]
\end{proof}

\section{Lambda-recursive sequences and the algebraic expansion matrix}\label{sec3}

\begin{definition}[Algebraic expansion matrix]
Fix $m\in\mathbb N_+$.
For $n\in\mathbb N$ and $k\in\mathbb N$ with $1\le k\le w_m(n)$, define the
\emph{algebraic expansion matrix} $\mathcal M_{(n,k)}=(a_{ij})_{1\le i,j\le k}$ by
\[
a_{ij}=
\begin{cases}
\lambda_1\!\bigl(n-(j-1)m,\; i-j+1\bigr), & \text{if } j\le i+1,\\[2pt]
0, & \text{if } j> i+1.
\end{cases}
\]
The family $\lambda_1$ is called the \emph{base sequence} of $\mathcal M_{(n,k)}$.
\end{definition}

By Theorem~\ref{thm4}, under its hypotheses, the coefficient $\lambda_3(n,k)$ satisfies
\[
\lambda_3(n,k)
=\;(-1)^k\,\frac{\det \mathcal M_{(n,k)}}{\displaystyle\prod_{i=0}^{k}\lambda_1(n-im,0)}.
\]

\begin{definition}[Lambda-recursive sequence of order $m$]\label{def:lamb-recur}
Let $m\in\mathbb N_+$ and let $\lambda_1:\mathbb N\times\mathbb Z\rightarrow R$
be a family with triangular support of order $m$, that is, $\lambda_1\in\Delta_m$.
The family $\lambda_1$ is said to be \emph{lambda-recursive (of order $m$)}
if there exist sequences $p_n\in R$ (principal factor) and $h_{(n,k)}\in R$ (auxiliary factor),
and initial data $c_n\in R$ for $0\le n\le m-1$, such that, for all admissible indices
$(n,k)$ with $0\le k\le w_m(n)$,
\[
\lambda_1(n,k)=
\begin{cases}
c_n, & 0\le n\le m-1 \text{ and } k=0,\\[4pt]
p_n\lambda_1(n-1,k)-h_{(n,k)}\lambda_1(n-m,k-1),
& n\ge m.
\end{cases}
\]
\end{definition}

\begin{theorem}[Lambda-recursive decomposition by residue classes]\label{thm:lamb-recur-classes}
Let $m\in\mathbb N_+$ and let
$\lambda_1:\mathbb N\times\mathbb Z\to R$
be a lambda-recursive sequence of order $m$ in the sense of
Definition~\textup{\ref{def:lamb-recur}}. For each residue
$r\in\{0,1,\dots,m-1\}$ and all $k\ge0$, write $n = mk + r$
and define, for $0\le t\le k$,
\[
\lambda_1^{\langle r\rangle}(k,t)
=
\lambda_1(mk + r,\,k - t).
\]
Then, for each $r$, the family
$\lambda_1^{\langle r\rangle}:\mathbb N\times\mathbb Z\to R$ satisfies
\[
\lambda_1^{\langle r\rangle}(k,t)=0
\quad\text{whenever}\quad t<0\ \text{or}\ t>k,
\]
and the classwise coefficients satisfy
\begin{align}\label{eq:lamb-classes}
\lambda_1^{\langle r\rangle}(k,t) &= \notag\\
&\hspace*{-4em}\begin{cases}
c_r, & k=0,\ t=0,\\[4pt]
p_{mk+r}\,\lambda_1^{\langle r-1\rangle}(k,t)
-
h_{(mk+r,k-t)}\,\lambda_1^{\langle r\rangle}(k-1,t),
& k\ge1,\ 0\le r\le m-1,
\end{cases}
\end{align}
onde, no caso $r=0$, o termo
$\lambda_1^{\langle r-1\rangle}(k,t)$ é interpretado como
\[
\lambda_1^{\langle -1\rangle}(k,t)
=
\lambda_1(mk-1,\,k-t)
=
\lambda_1\bigl(m(k-1)+(m-1),\,k-t\bigr)
=
\lambda_1^{\langle m-1\rangle}(k-1,t-1).
\]

In particular, solving the global lambda-recursive problem for
$\lambda_1(n,b)$ is equivalent to solving, in parallel, the $m$ coupled
lambda-recursive subproblems given by \eqref{eq:lamb-classes}, one for each
residue class $n\equiv r\pmod m$.
\end{theorem}

\begin{proof}
By Definition~\ref{def:suporte-triangular}, $\lambda_1\in\Delta_m$, that is,
$\lambda_1(n,b)=0$ whenever $b<0$ or $b>w_m(n)$, with
$w_m(n)=\lfloor n/m\rfloor$.

Fix a residue $r$ and $k\ge0$, write $n=mk+r$, and define
\[
\lambda_1^{\langle r\rangle}(k,t)
=
\lambda_1(mk+r,k-t).
\]
If $t<0$, then $k-t>k=w_m(mk+r)$; if $t>k$, then $k-t<0$. In both cases,
triangular support implies $\lambda_1^{\langle r\rangle}(k,t)=0$. Hence,
$\lambda_1^{\langle r\rangle}(k,t)=0$ whenever $t<0$ or $t>k$.

For $0\le n\le m-1$ one has $w_m(n)=0$, and writing $n=r$ forces $k=0$, so
\[
\lambda_1^{\langle r\rangle}(0,0)=\lambda_1(r,0)=c_r.
\]

Now assume $n\ge m$. By Definition~\ref{def:lamb-recur}, for every admissible $b$,
\begin{equation}\label{eq:lamb-global}
\lambda_1(n,b)
=
p_n\,\lambda_1(n-1,b)
\;-\;
h_{(n,b)}\,\lambda_1(n-m,b-1).
\end{equation}
Fix $r\in\{0,\dots,m-1\}$ and $k\ge1$, write $n=mk+r$, and let $0\le t\le k$.
Set $b=k-t$ in \eqref{eq:lamb-global} to obtain
\[
\lambda_1(mk+r,\,k-t)
=
p_{mk+r}\,\lambda_1(mk+r-1,\,k-t)
\;-\;
h_{(mk+r,\,k-t)}\,\lambda_1(mk+r-m,\,k-t-1).
\]
By definition, the left-hand side equals $\lambda_1^{\langle r\rangle}(k,t)$.

\medskip\noindent
\textit{Case $1\le r\le m-1$.}
Since $mk+r-1=mk+(r-1)$ and $mk+r-m=m(k-1)+r$, one has
\[
\lambda_1(mk+r-1,\,k-t)=\lambda_1^{\langle r-1\rangle}(k,t),
\qquad
\lambda_1(m(k-1)+r,\,k-t-1)=\lambda_1^{\langle r\rangle}(k-1,t).
\]
Substituting yields the recurrence in \eqref{eq:lamb-classes} for $1\le r\le m-1$.

\medskip\noindent
\textit{Case $r=0$.}
Here $n=mk$, so $mk-1=m(k-1)+(m-1)$ and $mk-m=m(k-1)$. Define
\[
\lambda_1^{\langle -1\rangle}(k,t)=\lambda_1(mk-1,\,k-t).
\]
Then
\[
\lambda_1^{\langle -1\rangle}(k,t)=\lambda_1^{\langle m-1\rangle}(k-1,t-1),
\qquad
\lambda_1(mk-m,\,k-t-1)=\lambda_1^{\langle 0\rangle}(k-1,t).
\]
Substituting these identities yields the recurrence in \eqref{eq:lamb-classes} for $r=0$
(with the stated interpretation of $\lambda_1^{\langle r-1\rangle}$).

Finally, since each $n\ge0$ admits a unique decomposition $n=mk+r$ with $0\le r<m$,
the identity $\lambda_1^{\langle r\rangle}(k,t)=\lambda_1(mk+r,\,k-t)$ shows that the global
family $\lambda_1(n,b)$ is completely determined by the $m$ classwise families
$\lambda_1^{\langle r\rangle}(k,t)$. This proves the final statement.
\end{proof}

\begin{definition}[Admissible lambda-recursive family]
A lambda-recursive family $\lambda_1$ is said to be \emph{admissible} if
\[
\lambda_1(n,0)\in R^\times \quad \text{for all } n\ge 0.
\]
Throughout the theory of inverse and intermediate kernels, admissible
lambda-recursive families are considered.
\end{definition}

\medskip
\noindent
Many classical systems admit expansions of the form
\[
f_n(x)=\sum_{k=0}^{w_m(n)}\lambda_1(n,k)\,x^{n-mk},
\]
with lambda-recursive coefficients. Table~\ref{tab:m1} summarizes the corresponding
$(p_n,h_{(n,k)})$ and typical initial values (for $R=\mathbb R$); normalization
adjustments may be required.

\begin{table}[h]
\centering
\caption{Examples of families with lambda-recursive structure}\label{tab:m1}
\small
\begin{tabular*}{\textwidth}{@{\extracolsep\fill}lccccl}
\toprule
\textbf{Family} & $\boldsymbol{m}$ & $\boldsymbol{w_m(n)}$
& $\boldsymbol{p_n}$ & $\boldsymbol{h_{(n,k)}}$ & \textbf{Initial values}\\
\midrule
Laguerre $L_n$ & $1$ & $n$ &
$-\dfrac{1}{n}$ & $\dfrac{k-2n}{n}$ &
$c_0=1$ \\
Chebyshev $T_n$ & $2$ & $\lfloor n/2\rfloor$ &
$2$ & $1$ &
$c_0=1,\ c_1=1$\\
Chebyshev $U_n$ & $2$ & $\lfloor n/2\rfloor$ &
$2$ & $1$ &
$c_0=1,\ c_1=2$\\
Legendre $P_n$ & $2$ & $\lfloor n/2\rfloor$ &
$\dfrac{2n-1}{n}$ & $\dfrac{n-1}{n}$ &
$c_0=1,\ c_1=1$\\
Hermite (physicists') $H_n$ & $2$ & $\lfloor n/2\rfloor$ &
$2$ & $2n-2$ &
$c_0=1,\ c_1=2$\\
Hermite (probabilists') $He_n$ & $2$ & $\lfloor n/2\rfloor$ &
$1$ & $n-1$ &
$c_0=1,\ c_1=1$\\
Lucas $V_n$ & $2$ & $\lfloor n/2\rfloor$ &
$1$ & $1$ &
$c_0=2,\ c_1=1$\\
Fibonacci $F_n$ & $2$ & $\lfloor n/2\rfloor$ &
$1$ & $1$ &
$c_0=0,\ c_1=1$\\
\bottomrule
\end{tabular*}
\end{table}

In general, whenever a polynomial family $\{p_n(x)\}\subset R[x]$
satisfies a linear three-term recurrence in the index $n$ (for each fixed $x$),
with suitable coefficients, its monomial coefficients can be organized into
a lambda-recursive kernel in the sense of Definition~\ref{def:lamb-recur}.
In the orthogonal case, it is classical that orthogonality implies a
three-term recurrence (see Szeg{\H{o}}~\cite{Szego1975} and Chihara~\cite{Chihara1978}),
and the same structure also encompasses non-orthogonal families (including
combinatorial examples) that satisfy linear recurrences in $n$.

\begin{definition}[Boundary sequence]
Let $\lambda_1=\{\lambda_1(n,k)\}$ be a lambda-recursive family of order $m$.
The \emph{boundary sequence} (or \emph{zero column}) is the sequence
\[
(\lambda_1(n,0))_{n\ge 0}.
\]
\end{definition}

\begin{theorem}[Boundary factorization]\label{theorem borda}
Let $\lambda_1=\{\lambda_1(n,k)\}$ be a lambda-recursive family of order $m$
with principal factor $(p_n)_{n\ge m}$. Then the boundary sequence
$(\lambda_1(n,0))_{n\ge 0}$ is completely determined by the initial values
$\lambda_1(0,0),\dots,\lambda_1(m-1,0)$ and by the sequence $(p_n)_{n\ge m}$.
Equivalently,
\[
\lambda_1(n,0)=\lambda_1(m-1,0)\,\prod_{i=m}^{n} p_i
\qquad(n\ge m).
\]
\end{theorem}

\begin{proof}
By the recurrence in the column $k=0$, for each $i\ge m$ one has
\[
\lambda_1(i,0)
= p_i\,\lambda_1(i-1,0)\;-\;h_{(i,0)}\,\lambda_1(i-m,-1).
\]
Since $\lambda_1(i-m,-1)=0$ by triangular support (as $-1<0$), it follows that
\[
\lambda_1(i,0)=p_i\,\lambda_1(i-1,0)\qquad(i\ge m).
\]
As $\lambda_1(i-1,0)\in R^\times$ by admissibility, one may rewrite
\[
p_i=\frac{\lambda_1(i,0)}{\lambda_1(i-1,0)}\qquad(i\ge m).
\]

Taking the product from $i=m$ to $i=n$ gives
\[
\prod_{i=m}^{n} p_i
=\prod_{i=m}^{n} \frac{\lambda_1(i,0)}{\lambda_1(i-1,0)}.
\]
By telescopic cancellation, this yields
\[
\prod_{i=m}^{n} p_i
=\frac{\lambda_1(n,0)}{\lambda_1(m-1,0)},
\]
that is,
\[
\lambda_1(n,0)=\lambda_1(m-1,0)\,\prod_{i=m}^{n} p_i\qquad(n\ge m),
\]
as claimed.
\end{proof}

\begin{theorem}\label{thm6}
If $\mathcal{M}_{(n,k)}$ is the algebraic expansion matrix and the base sequence $\lambda_{1}$ is
lambda-recursive with $h_{(n,k)}= h_n$ (independent of $k$), then, for $1 \leq k \leq w_m(n)$,
\begin{equation}\label{eq:Mnk-rec}
\begin{aligned}
|\mathcal{M}_{(n,k)}|
&= \Bigl(\prod_{j=1}^{k} p_{\,n-(j-1)m}\Bigr)\,|\mathcal{M}_{(n-1,k)}|\\
&\quad-\Bigl(\prod_{j=1}^{k-1} p_{\,n-(j-1)m}\Bigr)\,
   h_{\,n-(k-1)m}\,\lambda_{1}(n-km,0)\,|\mathcal{M}_{(n-1,k-1)}|.
\end{aligned}
\end{equation}
\end{theorem}

\begin{proof}
Work over a commutative ring. Write
\[
\mathcal{M}_{(n,k)}=[A_{1}\mid\dots\mid A_{k}]
\quad\text{and}\quad
\mathcal{M}_{(n-1,k)}=[B_{1}\mid\dots\mid B_{k}],
\]
where, for each $1\le j\le k$, the vectors $A_j,B_j\in R^{k}$ denote the $j$-th column
of the respective matrices. For $1\le j\le k$ and $1\le i\le k$, set
\[
\begin{aligned}
A_j(i)&=\lambda_{1}\bigl(n-(j-1)m,\ i-j+1\bigr),\\
B_j(i)&=\lambda_{1}\bigl(n-1-(j-1)m,\ i-j+1\bigr).
\end{aligned}
\]
Also define, for $1\le j\le k$ and $1\le i\le k$,
\[
C_j(i)=\lambda_1\bigl(n-jm,\ i-j\bigr),
\]
so that $C_j\in R^{k}$ is a column vector.

By lambda-recursivity, for $1\le j\le k$ and $1\le i\le k$,
\begin{align*}
\lambda_1\bigl(n-(j-1)m,\ i-j+1\bigr)
&= p_{\,n-(j-1)m}\,\lambda_1\bigl(n-1-(j-1)m,\ i-j+1\bigr)\\
&\quad - h_{\,n-(j-1)m}\,\lambda_1\bigl(n-jm,\ i-j\bigr).
\end{align*}
Equivalently,
\[
A_j(i) = p_{\,n-(j-1)m}\,B_j(i) - h_{\,n-(j-1)m}\,C_j(i).
\]
Thus, columnwise,
\begin{equation}\label{rec_det}
    A_j \;=\; p_{\,n-(j-1)m}\,B_j \;-\; h_{\,n-(j-1)m}\,C_j\qquad (1\le j\le k).
\end{equation}

Moreover, by the definition of $A_{j+1}$, one has $C_j=A_{j+1}$ for $1\le j\le k-1$.
By linearity of the determinant in the $j$-th column and using $C_j=A_{j+1}$, for $j=1,\dots,k-1$,
\begin{align*}
    |A_1,\dots,A_j,A_{j+1},\dots,A_k|
&= p_{\,n-(j-1)m}\,|A_1,\dots,B_j,A_{j+1},\dots,A_k|\\
&\quad- h_{\,n-(j-1)m}\,|A_1,\dots,C_j,A_{j+1},\dots,A_k|\\
&= p_{\,n-(j-1)m}\,|A_1,\dots,B_j,A_{j+1},\dots,A_k|,
\end{align*}
since the term involving $C_j$ has two equal columns ($j$ and $j{+}1$) and therefore vanishes.
Repeating this procedure for $j=1,\dots,k-1$ yields
\[
|\mathcal{M}_{(n,k)}|
=\Bigl(\prod_{j=1}^{k-1}p_{\,n-(j-1)m}\Bigr)\,|B_1,\dots,B_{k-1},A_k|.
\]

Applying \eqref{rec_det} to the last column $A_k$ (that is, with $j=k$) gives
\[
|B_1,\dots,B_{k-1},A_k|
= p_{\,n-(k-1)m}\,|B_1,\dots,B_k|
\;-\;h_{\,n-(k-1)m}\,|B_1,\dots,B_{k-1},C_k|.
\]
Multiplying by the previously extracted factors, one obtains
\begin{align*}
|\mathcal{M}_{(n,k)}|
&= \Bigl(\prod_{j=1}^{k}p_{\,n-(j-1)m}\Bigr)\,|\mathcal{M}_{(n-1,k)}|\\
&\quad-\Bigl(\prod_{j=1}^{k-1}p_{\,n-(j-1)m}\Bigr)\,
      h_{\,n-(k-1)m}\,|B_1,\dots,B_{k-1},C_k|.
\end{align*}
Since $\lambda_1(n-km,-1)=0$ by triangular support, it follows that
\[
C_k=(0,\dots,0,\lambda_1(n-km,0))^\top.
\]
Factoring $\lambda_1(n-km,0)$ from the last column and identifying the upper block with
$\mathcal{M}_{(n-1,k-1)}$, one obtains
\[
|B_1,\dots,B_{k-1},C_k|=\lambda_1(n-km,0)\,|\mathcal{M}_{(n-1,k-1)}|.
\]
Substituting into the previous expression yields \eqref{eq:Mnk-rec}.
\end{proof}

\begin{theorem}\label{thm7}
If $m\mid n$ and the base sequence of $\mathcal{M}_{(n,k)}$ is lambda-recursive, then
\[
\bigl|\mathcal{M}_{(n-1,\;w_m(n))}\bigr|=0.
\]
\end{theorem}

\begin{proof}
Fix $k=w_m(n)$ and consider
\[
\mathcal{M}_{(n-1,k)}=(B_{1},\dots,B_{k}),
\]
where, for $1\le i,j\le k$,
\[
B_j(i)=\lambda_1\bigl(n-1-(j-1)m,\;i-j+1\bigr).
\]
By triangular support, $\lambda_1(N,\ell)=0$ whenever $\ell>w_m(N)$.

The last row (that is, $i=k=w_m(n)$) is given by
\[
b_{w_m(n),\,j}
=\lambda_1\!\bigl(n-(j-1)m-1,\;w_m(n)-j+1\bigr),
\qquad 1\le j\le k.
\]
Since $m\mid n$, there exists $q\in\mathbb N$ such that $n=qm$, hence
\[
w_m(n)=\Big\lfloor \frac{n}{m}\Big\rfloor = \frac{n}{m}=q.
\]
For each $1\le j\le k$, one has
\[
\begin{aligned}
w_m\!\bigl(n-(j-1)m-1\bigr)
&=\Big\lfloor\frac{n-(j-1)m-1}{m}\Big\rfloor
 =\Big\lfloor q-(j-1)-\tfrac{1}{m}\Big\rfloor\\
&= q-j
 = w_m(n)-j.
\end{aligned}
\]
Therefore,
\[
w_m\!\bigl(n-(j-1)m-1\bigr)+1
= w_m(n)-j+1,
\]
and the column index in $b_{w_m(n),\,j}$ is exactly
$w_m\!\bigl(n-(j-1)m-1\bigr)+1$, that is,
\[
b_{w_m(n),\,j}
=\lambda_1\!\bigl(n-(j-1)m-1,\ w_m\!\bigl(n-(j-1)m-1\bigr)+1\bigr).
\]
By triangular support, $\lambda_1(N,\ell)=0$ whenever $\ell>w_m(N)$; hence,
\[
b_{w_m(n),\,j}=0\qquad(1\le j\le k).
\]
Thus the last row of $\mathcal{M}_{(n-1,k)}$ is zero, and therefore
\[
\bigl|\mathcal{M}_{(n-1,k)}\bigr|=0.\qedhere
\]
\end{proof}

\section{Basis inversion in lambda-recursive families}\label{sec:inv}

\medskip
From this point on, it is assumed that $\lambda_1$ is an admissible lambda-recursive family and that
\[
p_n\in R^\times \quad \text{for all } n\ge m.
\]
Thus, the fractions $1/\lambda_1(n,0)$, $1/p_t$, and $h_s/p_s$ appearing in the formulas below
are well defined in~$R$.

\begin{theorem}\label{thm8}
Let
\[
f_{n}(x)=\sum_{b=0}^{w_m(n)}\lambda_{1}(n,b)\,x^{\,n-mb},
\]
where $\lambda_{1}$ is lambda-recursive of order $m$ with principal factor
$(p_{n})_{n\ge m}$ and auxiliary factor independent of $k$, that is,
$h_{(n,k)}= h_{n}$. Assume that there exists a triangular inverse kernel
$\lambda_3\in\Delta_m$ such that, for all $n\in\mathbb N$,
\[
x^{n}=\sum_{b=0}^{w_m(n)}\lambda_{3}(n,b)\,f_{n-mb}(x),
\]
with normalization
\[
\lambda_1(n,0)\,\lambda_3(n,0)=1.
\]
For each pair $(n,k)$, write
\[
t=n-km,\qquad s=n-(k-1)m.
\]
Then, for all $n\in\mathbb N$ and all $k\in\mathbb Z$,
\[
\lambda_{3}(n,k)=
\begin{cases}
   \dfrac{1}{\lambda_{1}(n,0)}, & k=0,\\[8pt]
    A(n,k)\,\lambda_{3}(n-1,k)
    +\dfrac{h_{s}}{p_{s}}\;\lambda_{3}(n-1,k-1), & 1\le k\le w_m(n),
\end{cases}
\]
where
\[
A(n,k)=
\begin{cases}
\dfrac{1}{p_{\,t}}, & t\ge m,\\[6pt]
\dfrac{\lambda_{1}(t-1,0)}{\lambda_{1}(t,0)}, & t \in \{1,\dots,m-1\},\\[8pt]
0, & t=0.
\end{cases}
\]
\end{theorem}

\begin{proof}
For $k=0$, the normalization $\lambda_1(n,0)\lambda_3(n,0)=1$ implies
\[
\lambda_{3}(n,0)=\frac{1}{\lambda_{1}(n,0)}.
\]

For $1\le k\le w_m(n)$, use the determinantal expression of Theorem~\ref{thm4}:
\[
\lambda_{3}(n,k)=\frac{(-1)^k\,|\mathcal M_{(n,k)}|}
{\displaystyle\prod_{i=0}^{k}\lambda_{1}(n-im,0)}.
\]
Applying the recurrence for $|\mathcal M_{(n,k)}|$ (Theorem~\ref{thm6}) gives
\begin{align*}
\lambda_{3}(n,k)
&=\frac{(-1)^k}{\prod_{i=0}^{k}\lambda_{1}(n-im,0)}
\biggl(\Bigl[\prod_{i=1}^{k}p_{\,n-(i-1)m}\Bigr]|\mathcal M_{(n-1,k)}|\\[-2pt]
&\qquad-\Bigl[\prod_{i=1}^{k-1}p_{\,n-(i-1)m}\Bigr]
h_{\,n-(k-1)m}\,\lambda_{1}(n-km,0)\,|\mathcal M_{(n-1,k-1)}|\biggr).
\end{align*}

\medskip
\noindent\textbf{First term.}
By the boundary recurrence $\lambda_1(N,0)=p_N\,\lambda_1(N-1,0)$ for $N\ge m$, and since
$k\le w_m(n)$ implies $n-im\ge m$ for $0\le i\le k-1$, one has
\[
\prod_{i=0}^{k}\lambda_{1}(n-im,0)
=\Bigl(\prod_{i=0}^{k-1}p_{\,n-im}\Bigr)\,\lambda_{1}(n-km,0)\,
\Bigl(\prod_{i=0}^{k-1}\lambda_{1}(n-im-1,0)\Bigr).
\]
Hence,
\[
\frac{\displaystyle\prod_{i=1}^{k}p_{\,n-(i-1)m}}{\displaystyle\prod_{i=0}^{k}\lambda_{1}(n-im,0)}
=\frac{1}{\lambda_{1}(n-km,0)}\cdot
\frac{1}{\displaystyle\prod_{i=0}^{k-1}\lambda_{1}(n-im-1,0)}.
\]

Writing $t=n-km$, consider the following cases.

\smallskip
\emph{Case $t\ge m$.} One has $\lambda_{1}(t,0)=p_{\,t}\,\lambda_{1}(t-1,0)$ and
\[
\prod_{i=0}^{k}\lambda_{1}(n-1-im,0)
=\lambda_{1}(t-1,0)\,\prod_{i=0}^{k-1}\lambda_{1}(n-im-1,0).
\]
Thus,
\begin{align*}
\frac{\prod_{i=1}^{k}p_{\,n-(i-1)m}}{\prod_{i=0}^{k}\lambda_{1}(n-im,0)}
&=\frac{1}{\lambda_{1}(t,0)}\cdot\frac{1}{\prod_{i=0}^{k-1}\lambda_{1}(n-im-1,0)}\\
&=\frac{1}{p_t}\cdot\frac{1}{\prod_{i=0}^{k}\lambda_{1}(n-1-im,0)}.
\end{align*}
Therefore,
\[
\frac{(-1)^k}{\prod_{i=0}^{k}\lambda_{1}(n-im,0)}
\Bigl[\prod_{i=1}^{k}p_{\,n-(i-1)m}\Bigr]|\mathcal M_{(n-1,k)}|
=\frac{1}{p_{\,t}}\;\lambda_{3}(n-1,k).
\]

\smallskip
\emph{Case $t\in\{1,\dots,m-1\}$.} Here only
\[
\prod_{i=0}^{k}\lambda_{1}(n-1-im,0)
=\lambda_{1}(t-1,0)\,\prod_{i=0}^{k-1}\lambda_{1}(n-im-1,0)
\]
is used (since $n-1-km=t-1$). Then
\begin{align*}
\frac{\prod_{i=1}^{k}p_{\,n-(i-1)m}}{\prod_{i=0}^{k}\lambda_{1}(n-im,0)}
&=\frac{1}{\lambda_{1}(t,0)}\cdot\frac{1}{\prod_{i=0}^{k-1}\lambda_{1}(n-im-1,0)}\\
&=\frac{\lambda_{1}(t-1,0)}{\lambda_{1}(t,0)}\cdot
\frac{1}{\prod_{i=0}^{k}\lambda_{1}(n-1-im,0)}.
\end{align*}
Therefore,
\[
\frac{(-1)^k}{\prod_{i=0}^{k}\lambda_{1}(n-im,0)}
\Bigl[\prod_{i=1}^{k}p_{\,n-(i-1)m}\Bigr]|\mathcal M_{(n-1,k)}|
=\frac{\lambda_{1}(t-1,0)}{\lambda_{1}(t,0)}\;\lambda_{3}(n-1,k).
\]

\medskip
\noindent\textbf{Second term.}
For the term involving $|\mathcal M_{(n-1,k-1)}|$, one has
\[
\begin{aligned}
&-\frac{(-1)^{k}}{\prod_{i=0}^{k}\lambda_{1}(n-im,0)}
 \biggl(\prod_{i=1}^{k-1}p_{\,n-(i-1)m}\biggr)
 h_{\,n-(k-1)m}\,\lambda_{1}(n-km,0)\,|\mathcal M_{(n-1,k-1)}|\\
&\qquad=\frac{(-1)^{k+1}\bigl(\prod_{i=1}^{k-1}p_{\,n-(i-1)m}\bigr)
 h_{\,n-(k-1)m}\,\lambda_{1}(n-km,0)\,|\mathcal M_{(n-1,k-1)}|}
{\Bigl(\prod_{i=0}^{k-1}p_{\,n-im}\Bigr)\,\lambda_{1}(n-km,0)\,
\Bigl(\prod_{i=0}^{k-1}\lambda_{1}(n-im-1,0)\Bigr)}.
\end{aligned}
\]
Using
\(
\prod_{i=1}^{k-1}p_{\,n-(i-1)m}=\prod_{i=0}^{k-2}p_{\,n-im}
\),
this simplifies to
\[
\frac{(-1)^{k+1}h_{\,n-(k-1)m}\,|\mathcal M_{(n-1,k-1)}|}
     {p_{\,n-(k-1)m}\,\prod_{i=0}^{k-1}\lambda_{1}(n-im-1,0)}.
\]
Since
\[
\prod_{i=0}^{k-1}\lambda_{1}(n-1-im,0)
=\prod_{i=0}^{k-1}\lambda_{1}(n-im-1,0),
\]
the determinantal formula for $\lambda_3(n-1,k-1)$ yields
\[
\lambda_3(n-1,k-1)
=\frac{(-1)^{k-1}|\mathcal M_{(n-1,k-1)}|}
{\displaystyle\prod_{i=0}^{k-1}\lambda_1(n-1-im,0)}.
\]
Therefore,
\begin{align*}
&-\frac{(-1)^{k}}{\prod_{i=0}^{k}\lambda_{1}(n-im,0)}
 \biggl(\prod_{i=1}^{k-1}p_{\,n-(i-1)m}\biggr)
 h_{\,n-(k-1)m}\,\lambda_{1}(n-km,0)\,|\mathcal M_{(n-1,k-1)}|\\
&\qquad=\frac{h_{\,s}}{p_{\,s}}\;\lambda_{3}(n-1,k-1),
\end{align*}
where $s=n-(k-1)m$ (note that $s\ge m$, hence $p_s\in R^\times$ by assumption).

\medskip
\noindent\textbf{Case $t=0$.}
When $t=n-km=0$, one has $m\mid n$ and $k=w_m(n)$, hence $w_m(n-1)=k-1$. By triangular support,
$\lambda_3(n-1,k)=0$, so the contribution of the first term vanishes automatically. Defining $A(n,k)=0$
in this case yields a uniform statement.

\medskip
Collecting the three subcases for $t=n-km$ (namely, $t\ge m$, $t\in\{1,\dots,m-1\}$, and $t=0$),
the recurrence follows.
\end{proof}

\begin{theorem}[Symmetric change of basis]\label{thm11}
Fix $m\in\mathbb N_+$ and write $w_m(n)=\lfloor n/m\rfloor$.
Consider the families
\[
f_{n}(x)=\sum_{b=0}^{w_m(n)} \lambda_{1}(n,b)\,x^{\,n-mb},
\qquad
g_{n}(x)=\sum_{c=0}^{w_m(n)} \mu_{1}(n,c)\,x^{\,n-mc},
\]
with triangular support (zero entries outside $0\le k\le w_m(n)$) and initial columns
$\lambda_1(r,0),\mu_1(r,0)\in R^\times$ for $0\le r\le m-1$.
Assume that $\lambda_1$ and $\mu_1$ are \emph{lambda-recursive families of order $m$} with parameters
\[
\lambda_{1}(n,k)=
\begin{cases}
c^{(1)}_n, & 0\le n\le m-1,\ k=0,\\[2pt]
p^{(1)}_{n}\,\lambda_{1}(n-1,k)\;-\;h^{(1)}_{n}\,\lambda_{1}(n-m,k-1), & n\ge m,\ 0\le k\le w_m(n),
\end{cases}
\]
\[
\mu_{1}(n,k)=
\begin{cases}
c^{(2)}_n, & 0\le n\le m-1,\ k=0,\\[2pt]
p^{(2)}_{n}\,\mu_{1}(n-1,k)\;-\;h^{(2)}_{n}\,\mu_{1}(n-m,k-1), & n\ge m,\ 0\le k\le w_m(n),
\end{cases}
\]
where $p^{(i)}_{n},h^{(i)}_{n}\in R$ may depend on $n$ (but are independent of $k$), and invertibility is assumed whenever they appear in denominators.

Write $t=n-km$, $s=n-(k-1)m$ and define the boundary factors
\[
A_\mu(n,k)=
\begin{cases}
\dfrac{1}{p^{(2)}_{t}}, & t\ge m,\\
\dfrac{\mu_{1}(t-1,0)}{\mu_{1}(t,0)}, & 1\le t<m,\\
0, & t=0,
\end{cases}
\quad
A_\lambda(n,k)=
\begin{cases}
\dfrac{1}{p^{(1)}_{t}}, & t\ge m,\\
\dfrac{\lambda_{1}(t-1,0)}{\lambda_{1}(t,0)}, & 1\le t<m,\\
0, & t=0.
\end{cases}
\]

\medskip
\noindent\textbf{Direction $f\Rightarrow g$.}
Assume that there exists an inverse kernel $\mu_3=\{\mu_3(n,k)\}$ associated with $g$, that is, for all $n \in \mathbb{N}$,
\[
x^{n}=\sum_{k=0}^{w_m(n)} \mu_{3}(n,k)\,g_{n-mk}(x),
\qquad
\mu_{3}(n,0)=\frac{1}{\mu_{1}(n,0)}.
\]
Assume further that, for $1\le k\le w_m(n)$,
\[
\mu_{3}(n,k)=A_\mu(n,k)\,\mu_{3}(n-1,k)\;+\;\frac{h^{(2)}_{\,s}}{p^{(2)}_{\,s}}\,\mu_{3}(n-1,k-1),
\]
where $s=n-(k-1)m$.

Defining
\[
z(n,k)=\sum_{b=0}^{k}\lambda_{1}(n,b)\,\mu_{3}(n-mb,k-b),
\]
one obtains the change-of-basis expansion
\[
f_{n}(x)=\sum_{k=0}^{w_m(n)} g_{n-mk}(x)\,z(n,k),
\]
where $z=\{z(n,k)\}\in\Delta_m$ and $z(n,0)=\lambda_1(n,0)/\mu_1(n,0)$.
Moreover, for $1\le k\le w_m(n)$, the coefficients $z(n,k)$ satisfy the recurrence
\[
z(n,k)=
p^{(1)}_{n}\,A_\mu(n,k)\,z(n-1,k)
\;+\;
p^{(1)}_{n}\,\frac{h^{(2)}_{\,s}}{p^{(2)}_{\,s}}\,z(n-1,k-1)
\;-\;
h^{(1)}_{n}\,z(n-m,k-1).
\]

\medskip
\noindent\textbf{Direction $g\Rightarrow f$.}
Symmetrically, assume that there exists an inverse kernel $\lambda_3=\{\lambda_3(n,k)\}$ associated with $f$, that is, for all $n \in \mathbb{N}$,
\[
x^{n}=\sum_{k=0}^{w_m(n)} \lambda_{3}(n,k)\,f_{n-mk}(x),
\qquad
\lambda_{3}(n,0)=\frac{1}{\lambda_{1}(n,0)}.
\]
Assume further that, for $1\le k\le w_m(n)$,
\[
\lambda_{3}(n,k)=A_\lambda(n,k)\,\lambda_{3}(n-1,k)\;+\;\frac{h^{(1)}_{\,s}}{p^{(1)}_{\,s}}\,\lambda_{3}(n-1,k-1).
\]

Defining
\[
g_{n}(x)=\sum_{k=0}^{w_m(n)} f_{n-mk}(x)\,y(n,k),
\qquad
y(n,k)=\sum_{b=0}^{k}\mu_{1}(n,b)\,\lambda_{3}(n-mb,k-b),
\]
one again obtains a change-of-basis expansion, with $y=\{y(n,k)\}\in\Delta_m$
and $y(n,0)=\mu_1(n,0)/\lambda_1(n,0)$, and, for $1\le k\le w_m(n)$,
\[
y(n,k)=
p^{(2)}_{n}A_\lambda(n,k)y(n-1,k)
+
p^{(2)}_{n}\frac{h^{(1)}_{\,s}}{p^{(1)}_{\,s}}y(n-1,k-1)
-
h^{(2)}_{n}\,y(n-m,k-1).
\]
\end{theorem}

\begin{proof}
It suffices to prove the direction $f\Rightarrow g$, since the reverse direction follows by symmetry by exchanging
$(\lambda_1,p^{(1)}_{n},h^{(1)}_{n})$ with $(\mu_1,p^{(2)}_{n},h^{(2)}_{n})$.

Fix $\mu_3$ as the inverse kernel of $g$ and define $z(n,k)$ by
\[
f_{n}(x)=\sum_{k=0}^{w_m(n)} g_{n-mk}(x)\,z(n,k).
\]

One has
\[
f_{n}(x)=\sum_{b=0}^{w_m(n)} \lambda_{1}(n,b)\,x^{\,n-mb},
\quad
x^{n}=\sum_{c=0}^{w_m(n)} \mu_{3}(n,c)\,g_{n-mc}(x),
\]
with $\mu_{3}(n,0)=1/\mu_{1}(n,0)$.
Applying the second identity to each $x^{\,n-mb}$ in $f$ and reindexing by $k=b+c$
(Proposition~\ref{prop:conv-reindex}), one obtains
\[
\begin{aligned}
f_{n}(x)
&=\sum_{b=0}^{w_m(n)} \lambda_{1}(n,b)\!\!\sum_{c=0}^{w_m(n-mb)} \mu_{3}(n-mb,c)\,g_{n-m(b+c)}\bigl(x\bigr)\\
&=\sum_{k=0}^{w_m(n)} g_{n-mk}(x)\,\sum_{b=0}^{k}\lambda_{1}(n,b)\,\mu_{3}(n-mb,k-b)\\
&=\sum_{k=0}^{w_m(n)} g_{n-mk}(x)\,z(n,k),
\end{aligned}
\]
where
\[
z(n,k)=\sum_{b=0}^{k} \lambda_{1}(n,b)\,\mu_{3}(n-mb,k-b).
\]
By convention, $z(n,k)=0$ whenever $k<0$ or $k>w_m(n)$, and for $k=0$ one has
\[
z(n,0)=\lambda_{1}(n,0)\,\mu_{3}(n,0)=\frac{\lambda_{1}(n,0)}{\mu_{1}(n,0)}.
\]

Now fix $1\le k\le w_m(n)$ and use the recurrence for $\lambda_1$:
\[
\lambda_1(n,b)=p^{(1)}_{n}\lambda_1(n-1,b)-h^{(1)}_{n}\lambda_1(n-m,b-1),
\qquad 0\le b\le k.
\]
Substituting into $z(n,k)$ gives $z(n,k)=z_{1}(n,k)-z_{2}(n,k)$ with
\begin{align*}
    z_{1}(n,k)&=p^{(1)}_{n}\sum_{b=0}^{k}\lambda_1(n-1,b)\,\mu_3(n-mb,k-b),\\
    z_{2}(n,k)&=h^{(1)}_{n}\sum_{b=0}^{k}\lambda_1(n-m,b-1)\,\mu_3(n-mb,k-b).
\end{align*}

\emph{Second term.}
The term $b=0$ in $z_2$ vanishes since $\lambda_1(n-m,-1)=0$. With the change of index $c=b-1$,
\[
\begin{aligned}
z_{2}(n,k)
&=h^{(1)}_{n}\sum_{c=0}^{k-1}\lambda_1(n-m,c)\,\mu_3\bigl(n-m(c+1),k-(c+1)\bigr)\\
&=h^{(1)}_{n}\sum_{c=0}^{k-1}\lambda_1(n-m,c)\,\mu_3\bigl((n-m)-mc,(k-1)-c\bigr)\\
&=h^{(1)}_{n}\,z(n-m,k-1).
\end{aligned}
\]

\emph{First term.}
Let $s=n-(k-1)m$. Separate the term $b=k$ and treat the terms $b\le k-1$ via the recurrence for $\mu_3$.
For $t=n-km$, note that
\[
A_\mu(n,k)=A_\mu(n-mb,k-b),
\qquad
\frac{h^{(2)}_{\,s}}{p^{(2)}_{\,s}}
=\frac{h^{(2)}_{\, (n-mb)-((k-b)-1)m}}{p^{(2)}_{\, (n-mb)-((k-b)-1)m}},
\]
so these quantities are independent of $b$. Hence
\[
\begin{aligned}
z_{1}(n,k)
&=p^{(1)}_{n}\lambda_1(n-1,k)\,\mu_3(n-mk,0)\\
&\quad+p^{(1)}_{n}A_\mu(n,k)\sum_{b=0}^{k-1}\lambda_1(n-1,b)\,\mu_3(n-mb-1,k-b)\\
&\quad+p^{(1)}_{n}\frac{h^{(2)}_{\,s}}{p^{(2)}_{\,s}}
\sum_{b=0}^{k-1}\lambda_1(n-1,b)\,\mu_3(n-mb-1,k-b-1).
\end{aligned}
\]
The sums are identified as
\[
\sum_{b=0}^{k-1}\lambda_1(n-1,b)\,\mu_3(n-mb-1,k-b)
=z(n-1,k)-\lambda_1(n-1,k)\,\mu_3(n-mk-1,0),
\]
\[
\sum_{b=0}^{k-1}\lambda_1(n-1,b)\,\mu_3(n-mb-1,k-b-1)=z(n-1,k-1).
\]
Therefore,
\[
\begin{aligned}
z_{1}(n,k)
&=p^{(1)}_{n}A_\mu(n,k)\,z(n-1,k)
+p^{(1)}_{n}\frac{h^{(2)}_{\,s}}{p^{(2)}_{\,s}}\,z(n-1,k-1)\\
&\quad+p^{(1)}_{n}\lambda_1(n-1,k)\Bigl(\mu_3(n-mk,0)-A_\mu(n,k)\,\mu_3(n-mk-1,0)\Bigr).
\end{aligned}
\]

The parenthetical term vanishes. Write again $t=n-km$.

\smallskip
\emph{Case $t\ge m$.}
Since $\mu_1(t,0)=p^{(2)}_{t}\,\mu_1(t-1,0)$, one has
\[
\mu_3(t,0)=\frac{1}{\mu_1(t,0)}
=\frac{1}{p^{(2)}_{t}}\,\mu_3(t-1,0)
=A_\mu(n,k)\,\mu_3(t-1,0),
\]
so the parenthetical term is zero.

\smallskip
\emph{Case $1\le t\le m-1$.}
By definition $A_\mu(n,k)=\mu_1(t-1,0)/\mu_1(t,0)$ and, since
$\mu_3(t,0)=1/\mu_1(t,0)$ and $\mu_3(t-1,0)=1/\mu_1(t-1,0)$, one again obtains
$\mu_3(t,0)=A_\mu(n,k)\,\mu_3(t-1,0)$.

\smallskip
\emph{Case $t=0$.}
Then $n=km$, hence $k=w_m(n)$ and $w_m(n-1)=k-1$, so $k>w_m(n-1)$; by triangular support,
$\lambda_1(n-1,k)=0$, and the remaining term is automatically zero.

\smallskip
Consequently,
\[
z_{1}(n,k)
=
p^{(1)}_{n}A_\mu(n,k)\,z(n-1,k)
+
p^{(1)}_{n}\frac{h^{(2)}_{\,s}}{p^{(2)}_{\,s}}\,z(n-1,k-1).
\]

\medskip
\emph{Conclusion.}
Since $z(n,k)=z_1(n,k)-z_2(n,k)$, it follows that, for $1\le k\le w_m(n)$,
\[
z(n,k)=
p^{(1)}_{n}\,A_\mu(n,k)\,z(n-1,k)
\;+\;
p^{(1)}_{n}\,\frac{h^{(2)}_{\,s}}{p^{(2)}_{\,s}}\,z(n-1,k-1)
\;-\;
h^{(1)}_{n}\,z(n-m,k-1),
\]
as claimed.
The direction $g\Rightarrow f$ is obtained by symmetrically exchanging
\[
(\lambda_1,p^{(1)}_{n},h^{(1)}_{n},A_\lambda)\quad\longleftrightarrow\quad(\mu_1,p^{(2)}_{n},h^{(2)}_{n},A_\mu),
\]
which reproduces the formulas in the statement.
\end{proof}

\section{Conclusions and perspectives}\label{sec:conc}

This work develops a unified framework for polynomial families
\[
f_{n}(x)=\sum_{b=0}^{w_m(n)}\lambda_1(n,b)\,x^{n-mb}
\]
with triangular support of order $m$ over a commutative ring $R$. Starting from a direct kernel $\lambda_1$ and the function $w_m$, we establish a triangular inversion theorem (\Cref{thm:inversao-direta-base}) that characterizes—via a discrete orthogonality condition—when an inverse kernel $\lambda_3$ exists implementing the change of basis $x^n\leftrightarrow f_{n}(x)$. Uniqueness of inversion coefficients is proved, determinantal formulas are obtained via the algebraic expansion matrix $\mathcal M_{(n,k)}$, and a composition mechanism is described for inversions and basis changes, including across distinct orders $m_1,m_2$.

A second main contribution is the introduction of lambda-recursive sequences of order $m$ and the resulting triangular structure. Given a principal factor $(p_n)$ and auxiliary factors $(h_{(n,k)})$, the direct kernel $\lambda_1(n,k)$ is controlled by boundary data and classwise parameters: the boundary sequence $\lambda_1(n,0)$ is determined by $(p_n)$ (\Cref{theorem borda}), and local triangular recurrences are derived for the inverse-kernel coefficients $\lambda_3(n,k)$ (\Cref{thm8}). This makes transparent how variations in boundary data and in the ratios $h_n/p_n$ propagate through determinantal formulas and recursive updates.

From a structural viewpoint, the results indicate that much of the information of a lambda-recursive family is concentrated in boundary data and classwise parameters, while the triangular expansion matrix canonically organizes compatible normalizations of the same inversion mechanism. This perspective recovers several classical examples—including orthogonal polynomials and combinatorial sequences—and provides systematic, mechanically derived recurrences for connection and change-of-basis coefficients.

Regarding perspectives, the framework suggests both algebraic and computational developments. Algebraically, triangular kernels and the expansion matrix can be viewed as universal operators encoding connection coefficients between adapted bases. Computationally, the classwise recurrences for $\lambda_3(n,k)$ and for basis-change coefficients suggest structured generalizations of Clenshaw-type schemes for recursive bases of order $m$, with potential applications to evaluation, interpolation, and connections of orthogonal systems. These directions—covering triangular convolution algebras, connection coefficients, and structured Clenshaw-type schemes—will be explored in future work.

\bibliographystyle{amsplain}
\bibliography{refs}

\end{document}